\documentclass[12pt,a4paper]{amsart}

\usepackage[utf8]{inputenc}
\usepackage[english]{babel}
\usepackage[T1]{fontenc}
\usepackage{bbm}
\usepackage{hyperref}
\usepackage{enumitem}
\setenumerate[1]{label=(\roman{*}), ref=\roman{*}}
\setenumerate[2]{label=(\alph{*}), ref=\alph{*}}

\theoremstyle{plain}
\newtheorem*{theorem*}{\indent Theorem}
\newtheorem{theorem}{\indent Theorem}[section]
\newtheorem{corollary}[theorem]{\indent Corollary}

\newtheorem{proposition}[theorem]{\indent Proposition}
\newtheorem{question}[theorem]{\indent Question}

\theoremstyle{definition}

\newtheorem*{example}{\indent Example}

\theoremstyle{remark}
\newtheorem{rem}[theorem]{\indent Remark}

\DeclareMathOperator{\spn}{span}
\DeclareMathOperator{\conv}{conv}

\newcommand{\xast}{x^{\ast}}
\newcommand{\Xast}{X^{\ast}}

\newcommand{\iten}{\ensuremath{\widehat{\otimes}_\varepsilon}}
\newcommand{\pten}{\ensuremath{\widehat{\otimes}_\pi}}

\title[ASQ and OH norms in tensor product spaces]{Almost square and octahedral norms in tensor products of Banach spaces}
\author[J. Langemets, V. Lima and A. Rueda Zoca]{Johann Langemets, Vegard Lima and Abraham Rueda Zoca}

\address{Institute of Mathematics and Statistics, University of Tartu, J. Liivi 2, 50409 Tartu, Estonia} \email{johann.langemets@ut.ee}
\thanks{The research of J.~Langemets was supported by institutional research funding IUT20-57 of the Estonian Ministry of Education and Research.}

\address{NTNU in Ålesund, Postboks 1517, 6025 Ålesund, Norway} \email{Vegard.Lima@gmail.com}

\address{Universidad de Granada, Facultad de Ciencias.
Departamento de An\'{a}lisis Matem\'{a}tico, 18071-Granada
(Spain)} \email{arz0001@correo.ugr.es}\thanks{Third author was partially supported by Junta de Andaluc\'ia Grants FQM-0199.}

\begin{document}

\maketitle

\begin{abstract}
The aim of this note is to study some geometrical properties like
diameter two properties, octahedrality and almost squareness in the
setting of (symmetric) tensor product spaces.
In particular, we show that the injective tensor product
of two octahedral Banach spaces is always octahedral,
the injective tensor product of an almost square Banach
space with any Banach space is almost square,
and the injective symmetric tensor product of
an octahedral Banach space is octahedral.
\end{abstract}

\section{Introduction}
Let $X$ be a (real) Banach space
with closed unit ball $B_X$ and unit sphere $S_X$.
Following \cite{aln}, we will say that
\begin{enumerate}
\item
  $X$ has the \emph{local diameter two property} (LD2P) whenever
  each slice of $B_X$ has diameter two.
\item
  $X$ has the \emph{diameter two property} (D2P) whenever each
  non-empty relatively weakly open subset of $B_X$ has diameter
  two.
\item
  $X$ has the \emph{strong diameter two property} (SD2P)
  whenever each finite convex combination of slices of $B_X$ has
  diameter two.
\end{enumerate}
Similarly, for dual spaces one can define the \emph{$w^*$-LD2P},
the \emph{$w^*$-D2P} and the \emph{$w^*$-SD2P} by replacing slices and
weakly open subsets with weak$^*$ slices and weak$^*$ open subsets
in the above definitions.

The starting point for the study of diameter two properties
was probably \cite{nw}.
Recent years have seen a lot of activity in
the study of diameter two properties, see
\cite{ahntt,aln,blrext,hlp} and the references therein.

In order to characterize the dual of spaces with
diameter two properties Haller, Langemets and P{\~o}ldvere
studied octahedral norms in \cite{hlp}.
There they proved the following characterizations
of octahedral norms that we will take as our definitions.
A Banach space $X$ is
\begin{enumerate}
\item
  \emph{locally octahedral} (LOH) if for every $x\in S_X$
  and $\varepsilon>0$ there exists $y\in S_X$ such that
  $\Vert x\pm y\Vert>2-\varepsilon$.
\item
  \emph{weakly octahedral} (WOH) if for every
  $x_1,\ldots, x_n\in S_X$, $x^*\in B_{X^*}$ and $\varepsilon>0$
  there exists $y\in S_X$ such that
  $\Vert x_i\pm ty\Vert\geq (1-\varepsilon)(\vert x^*(x_i)\vert +t)$
  for all $i\in\{1,\ldots, n\}$ and $t>0$.
\item
  \emph{octahedral} (OH) if for every $x_1,\ldots, x_n\in S_X$ and
  $\varepsilon>0$ there exists $y\in S_X$ such that
  $\Vert x_i+y\Vert>2-\varepsilon$ for all $i\in\{1,\ldots, n\}$.
\end{enumerate}

It is known that a Banach space $X$ has the SD2P (respectively D2P, LD2P)
if, and only if, $X^*$ is OH (respectively WOH, LOH) (cf.\ e.g.\ \cite{hlp}).

A family of geometrical properties closely related to
diameter two properties and octahedrality is the following
which was introduced in \cite{all}.
A Banach space $X$ is
\begin{enumerate}
\item
  \emph{locally almost square} (LASQ) if for every $x\in S_X$
  there exists a sequence $\{y_n\}$ in $B_X$ such that
  $\Vert x\pm y_n\Vert\rightarrow 1$ and $\Vert y_n\Vert\rightarrow 1$.
\item
  \emph{weakly almost square} (WASQ) if for every $x\in S_X$
  there exists a sequence $\{y_n\}$ in $B_X$ such that
  $\Vert x\pm y_n\Vert\rightarrow 1$, $\Vert y_n\Vert\rightarrow 1$
  and $y_n \rightarrow 0$ weakly.
\item
  \emph{almost square} (ASQ) if for every $x_1,\ldots, x_k \in S_X$
  there exists a sequence $\{y_n\}$ in $B_X$ such
  that $\Vert y_n\Vert\rightarrow 1$ and $\Vert x_i\pm y_n\Vert\rightarrow 1$
  for every $i\in\{1,\ldots, k\}$.
\end{enumerate}

It is known that the sequence involved in the definition of ASQ
can be chosen to be weakly-null \cite[Theorem~2.8]{all}, so
ASQ implies WASQ which in turn implies LASQ\@.
Moreover, ASQ implies the SD2P,
WASQ implies the D2P, and LASQ implies the LD2P
(see \cite{all}, but note that
the latter two statements were proved in \cite{kub}).

Diameter two properties of (symmetric) tensor products,
both injective and projective, have attracted the attention
of many researchers. In \cite[Question (b)]{aln}, it was explicitly posed as an open question
how diameter two properties are preserved by tensor products.
Partial answers,
which strongly rely on infinite-dimensional centralizers,
appeared in \cite{abr} and \cite{ab2}. For instance,
it is known that the projective tensor product
of a $C(K)$ space with any non-zero Banach space has the D2P
\cite[Theorem~4.1]{abr} and that the symmetric projective tensor
product of any $L_1(\mu)$ space has the D2P \cite[Theorem~3.3]{ab2}.
However, the assumption of having infinite-dimensional centralizer
have been shown to be far from necessary. In the symmetric case it has been recently proved that the symmetric projective tensor
product of any ASQ space has the SD2P \cite[Theorem~3.3]{blr3}.
Furthermore, in the non-symmetric case there are even stability
results for some diameter two properties, e.g.\ both the LD2P and the SD2P
are stable by taking projective tensor product
(see \cite[Theorem~2.7]{aln} and \cite[Corollary~3.6]{blr}).
In spite of the previous nice results,
the interplay between diameter two properties, octahedrality,
and almost squareness with respect to tensor products is
currently not well understood. Thus, the aim of this note is to go further and study
octahedrality and almost squareness and their
relations to (symmetric) tensor products
equipped with the injective or projective norm.

In Section~\ref{tensorproducts} we start by
showing that octahedrality is stable by taking injective
tensor products, see Theorem~\ref{theoctaopera} and
Corollary~\ref{estaoctainje}.
Our results on octahedrality of injective tensor products
are refinements of the results of
\cite{blr}, where octahedrality of spaces of operators
were deeply studied.
Next we turn to almost squareness.
A consequence of Theorem~\ref{asqoperator}
is that the injective tensor product of an ASQ space
with any non trivial Banach space is ASQ\@.
The projective norm is much more difficult to work
with, but we are able to show that the projective
tensor product of $c_0$ with any non-zero Banach space is LASQ
in Proposition~\ref{tensorc0ASQ}.

In Section~\ref{symtenprod} we study the symmetric tensor products.
Our main result is Theorem~\ref{estaoctainjesyme} which is a stability
result of octahedrality for injective symmetric tensor products.
It will allow us to significantly improve some results from \cite[Section~4]{ab}
(see Corollary~\ref{integralpolyD2P}).
Combining WASQ and the Dunford-Pettis property
we get LD2P of the projective symmetric tensor product
in Proposition~\ref{dunfordpettistensoprosym}.
We will also give an inheritance result
for WASQ spaces which, combined with the above
result, will result in a criterion for
the symmetric projective tensor product of
a subspace of $L_1(\mu)$ to have the LD2P\@.

We close the paper with some open questions in Section~\ref{remarks}.

We use standard Banach space notation, as can be found in e.g.\ \cite{alka}.
Given Banach spaces $X$ and $Y$, $L(X,Y)$ (respectively $K(X,Y)$)
will denote the space of bounded linear operators
(respectively compact operators) from $X$ to $Y$.
We will consider only non-zero Banach spaces.

\section{Tensor product spaces}
\label{tensorproducts}

Recall that given two Banach spaces $X$ and $Y$, the
\textit{injective tensor product} of $X$ and $Y$, denoted by
$X \iten Y$, is the completion of $X\otimes Y$ under the norm given by
\begin{equation*}
   \Vert u\Vert:=\sup
   \left\{
      \sum_{i=1}^n \vert x^*(x_i)y^*(y_i)\vert
      : x^*\in S_{X^*}, y^*\in S_{Y^*}
   \right\},
\end{equation*}
where $u:=\sum_{i=1}^n x_i\otimes y_i$.
Every $u \in X \iten Y$ can be viewed as an operator $T_u : X^* \rightarrow Y$ which is weak$^*$-to-weakly continuous.
In particular, $y^* \circ T_u$ belongs to $X$ for all $y^* \in Y^*$ (see \cite{rya} for background).

Quite a lot is known about octahedrality of spaces of operators.
It is essentially known that the injective
tensor product $X \iten Y$ of two Banach spaces
is LOH if one of the spaces is LOH
(see the comment following Lemma~2.3 in \cite{blr} or \cite[Theorem~3.39]{lan}).
Hence it is natural to wonder when an
injective tensor product is OH\@.
Before we can state our results about octahedrality of $X \iten Y$ we will
need to introduce a bit of notation.
Given a Banach space $X$ and a norm one element $u$ in $X$, define
\begin{equation*}
   D(X,u):=\{f\in B_{X^*}: f(u)=1\}.
\end{equation*}
Define $n(X,u)$ as the largest non-negative real number $k$ satisfying
\begin{equation*}
   k\Vert x\Vert\leq 
   \sup\{ \vert f(x)\vert: f\in D(X,u)\}
\end{equation*}
for all $x\in X$ (see the discussion following Theorem~3.6
in \cite{abr} for examples of spaces which have such unitaries).
Note that $n(X,u)=1$ if, and only if, $D(X,u)$ is a norming subset for $X$.
We will also need the following geometrical characterization of
octahedrality that is proved in \cite{lan}.
\begin{theorem}[{\cite[Theorem~3.21]{lan}}]\label{charaOHlanthe}
Let $X$ be a Banach space. $X$ is OH if, and only if, whenever
$E$ is a finite-dimensional subspace of $X$,
$x_1^*,\ldots, x_n^*\in B_{X^*}$, $\varepsilon>0$ and
$0<\varepsilon_0<\varepsilon$ there exists $y\in S_X$ such that,
whenever $\vert\gamma_i\vert\leq 1+\varepsilon_0$, there exists
$y_i^*\in X^*$ satisfying
$\Vert y_i^*\Vert\leq 1+\varepsilon$
for all $i\in\{1,\ldots, n\}$,
$y^*_{i}|_E = x^*_{i}|_E$, and
$y_i^*(y)=\gamma_i$.
\end{theorem}

Our first theorem is a version of
\cite[Theorem~3.5]{blr} and \cite[Theorem~3.1]{blr}
stated in the context of weak$^*$-to-weakly continuous operators.
We include a proof that uses Theorem~\ref{charaOHlanthe}
instead of working with the $w^*$-SD2P of the dual space.

\begin{theorem}\label{theoctaopera}
Let $X$ and $Y$ be Banach spaces and let $H\subseteq L(X^*,Y)$ be a closed subspace
such that $X\otimes Y\subseteq H$. Assume that each $T\in H$ is
weak$^*$-to-weakly continuous.

\begin{enumerate}
\item\label{mejooctaopera} If $X$ and $Y$ are OH, then $H$ is OH\@.
\item\label{corounillocta} If $X$ is OH and there exists $y\in S_{Y}$ such that
$n(Y,y)=1$, then $H$ is OH\@.
\end{enumerate}

\end{theorem}

\begin{proof}
  (\ref{mejooctaopera}).
  Let $T_1,\ldots,T_k \in S_H$ and $\varepsilon > 0$.
  For each $i$ find $y_i^* \in S_{Y^*}$ and $x_i^* \in S_{X^*}$
  such that
  \begin{equation*}
    y_i^*(T_i x_i^*) = x_i^* (T_i^*y_i^*) > 1 - \varepsilon.
  \end{equation*}
  
  Let $E = \spn\{T_i^* y_i^* : i \in \{1,\ldots,k\}\} \subset X$.
  By Theorem~\ref{charaOHlanthe} we find
  $w \in S_X$ and $w_i^* \in X^*$, $i \in \{1,\ldots,k\}$,
  such that $w_i^*(T_i^* y_i^*) = x_i^*(T_i^* y_i^*)$,
  $w_i^*(w) = 1$ and $\|w_i^*\| \le 1+\varepsilon$.

  Let $F = \spn\{T_i w_i^* : i \in \{1,\ldots,k\}\} \subset Y$.
  Using Theorem~\ref{charaOHlanthe} again we find
  $z \in S_Y$ and $z_i^* \in Y^*$, $i \in \{1,\ldots,k\}$,
  such that $z_i^*(T_i w_i^*) = y_i^*(T_i w_i^*)$,
  $z_i^*(z) = 1$ and $\|z_i^*\| \le 1 + \varepsilon$.
  
  Define $S = w \otimes z \in X \otimes Y$. We have $S \in S_H$
  and
  \begin{align*}
    \|T_i + S\|
    &\ge \frac{1}{(1+\varepsilon)^2}
      z_i^*(T_i w_i^* + Sw_i^*)
    =  \frac{1}{(1+\varepsilon)^2}
      (y_i^*(T_i w_i^*) + w_i^*(w)z_i^*(z)) \\
    &=  \frac{1}{(1+\varepsilon)^2}
      (y_i^*(T_i x_i^*) + 1)
    > \frac{2-\varepsilon}{(1+\varepsilon)^2}.
  \end{align*}
 Consequently $H$ is OH\@.
 
 (\ref{corounillocta}).
 Let $T_1,\ldots,T_k \in S_H$ and $\varepsilon > 0$.
 For each $i$ find $y_i^* \in S_{Y^*}$ and $x_i^* \in S_{X^*}$
 such that
 \begin{equation*}
    y_i^*(T_i x_i^*) = x_i^* (T_i^*y_i^*) > 1 - \varepsilon.
 \end{equation*}
 By assumption $D(Y,y) = \{y^*\in S_{Y^*}: y^*(y)=1\}$ is norming for $Y$
 so we may assume that $y_i^*(y) = 1$ for each $i\in\{1,\ldots, n\}$.
 We now proceed as in (\ref{mejooctaopera}) by finding
 $w \in S_X$ and $w_i^* \in X^*$. We can then use $z_i^* = y_i^*$ and $z = y$ to conclude the proof as above.
\end{proof}

\begin{rem}
The assumption that every operator in $H$ is weak$^*$-to-weakly continuous
is necessary in the second statement of Theorem~\ref{theoctaopera}.
Indeed,
$Z:=L(\mathcal C[0,1]^*,\ell_\infty)=L(\ell_1,\mathcal C[0,1]^{**})$
is not OH from \cite[Corollary~3.12]{blr}.
\end{rem}

Since $H = X \iten Y$ is a subspace of $L(X^*,Y)$,
Theorem~\ref{theoctaopera} says the following
about octahedrality of injective tensor products.

\begin{corollary}\label{estaoctainje}
Let $X$ and $Y$ be Banach spaces.
\begin{enumerate}
\item If both $X$ and $Y$ are OH, then $X\iten Y$ is OH\@.
\item If $X$ is OH and there exists $y\in S_Y$ such that $n(Y,y)=1$, then $X\iten Y$ is OH\@.
\end{enumerate}
\end{corollary}

Let us also note the following necessary conditions
for injective tensor product spaces to be OH, WOH or LOH\@.
Proof of the below proposition follows
the ideas in \cite[Proposition~3.9]{blr} and
\cite[Theorem~3.46]{lan} and will not be included.
A Banach space has a \emph{non-rough norm}
if the dual unit ball has weak$^*$ slices of arbitrarily small diameter
(see \cite[Proposition~II.1.11]{dgz}).
\begin{proposition}
  Let $X$ and $Y$ be Banach spaces. Assume that $Y$ has a non-rough norm.
  \begin{enumerate}
  \item
    If $X \iten Y$ is LOH, then $X$ is LOH\@.
  \item
    If $X \iten Y$ is WOH, then $X$ is WOH\@.
  \item
    If $X \iten Y$ is OH, then $X$ is OH\@.
  \end{enumerate}
\end{proposition}

From octahedrality we now turn to almost squareness and
diameter two properties.
Acosta, Becerra Guerrero and Rodr\'iguez-Palacios have shown that if $Y$
is a Banach space, then $X \iten Y$
has the D2P whenever $X$ is a Banach space such that
the supremum of the dimension of the centralizer of all the even duals
of $X$ is unbounded \cite[Theorem~5.3]{abr}.
Now, we prove a stability result for
ASQ, which will provide a wide class
of injective tensor product spaces which have the SD2P\@.

\begin{theorem}\label{asqoperator}
Let $X$ and $Y$ be Banach spaces. Assume that $X$ is ASQ.
\begin{enumerate}
\item\label{item:asqop-1}
  If $H \subseteq K(Y,X)$ is a closed subspace with $Y^* \otimes X \subset H$, then $H$
  is ASQ\@.
\item\label{item:asqop-2}
  If $H \subseteq K(Y^*,X)$ is a closed subspace with $Y \otimes X \subset H$, then $H$
  is ASQ\@.
\end{enumerate}
\end{theorem}

\begin{proof}
We will prove only (\ref{item:asqop-2}).
The proof of (\ref{item:asqop-1}) is similar.

Let $T_1,\ldots,T_k \in S_H$ and $\varepsilon > 0$.
The set
\begin{equation*}
  K = \cup_{i=1}^k T_i(B_{Y^*})
\end{equation*}
is a relatively compact subset of $B_X$ hence
there exists $x_1,\ldots,x_n \in B_X$ such that
$K \subset \cup_{i=1}^n B(x_i,\varepsilon)$.
Define
\begin{equation*}
  E:=\spn\{x_{i}: i\in\{1,\ldots, n\}\}.
\end{equation*}
As $E$ is a finite-dimensional subspace of $X$ and $X$ is ASQ there exists $z\in S_X$ such that
\begin{equation*}
  \Vert e+\lambda z\Vert \leq
  (1+\varepsilon) \max\{ \Vert e\Vert, \vert \lambda\vert \}
\end{equation*}
for all $e\in E$ and $\lambda \in \mathbb{R}$ by \cite[Theorem~2.4]{all}.
Pick $y\in S_Y$, and define $S:=z \otimes y \in S_H$.
For each $T_i$ and $y^* \in B_{Y^*}$ there exists
$j$ such that $\|T_iy^* - x_j\| < \varepsilon$,
so
\begin{align*}
  \|(T_i + S)y^*\| &\le \|T_i y^* - x_j\|
  + \|x_j + y^*(y)z\| \\
  &\le \varepsilon + (1+\varepsilon) \max\{ \|x_j\|,|y^*(y)| \}
  \le 1+ 2\varepsilon.
\end{align*}
Taking supremum over all such $y^*$
we get $\|T_i + S\| \le 1+ 2\varepsilon$,
which means that $H$ is ASQ \cite[Proposition~2.1]{all}.
\end{proof}

\begin{example}
The converse of Theorem~\ref{asqoperator} does not hold.
Indeed, given $1< p\leq q<\infty$ then $K(\ell_p,\ell_q)$
is an $M$-ideal in its bidual $L(\ell_p,\ell_q)$ \cite[Example~VI.4.1]{hww} and, consequently,
it is ASQ \cite[Corollary~4.3]{all}.
Thus, $K(Y,X)$ can be ASQ even if both $X$ and $Y$ are uniformly convex.
\end{example}

\begin{corollary}\label{estasqinject}
Let $X$ and $Y$ be Banach spaces. If $X$ is ASQ, then $X\iten Y$ is ASQ\@. In
particular, $X \iten Y$ has the SD2P\@.
\end{corollary}

Using the above corollary we can characterize the $\mathcal C(K,X)$ spaces which are ASQ\@.

\begin{corollary}\label{caraoctacontivector}
Let $K$ be compact Hausdorff space and $X$ be a Banach space.
The following assertions are equivalent:
\begin{enumerate}
\item\label{item:CKASQ-1} $\mathcal C(K,X)$ is ASQ\@.
\item\label{item:CKASQ-2} $X$ is ASQ\@.
\end{enumerate}

\end{corollary}

\begin{proof}
(\ref{item:CKASQ-1}) $\Rightarrow$ (\ref{item:CKASQ-2}).
Pick $x_1,\ldots, x_n\in S_X$ and $\varepsilon>0$,
and let us find $x\in S_X$ such that $\Vert x_i\pm x\Vert\leq 1+\varepsilon$
for all $i\in\{1,\ldots, n\}$. For each $i\in\{1,\ldots, n\}$ define
\begin{equation*}
  f_i(t)=x_i \text{ for all } t\in K,
\end{equation*}
which is an element of $S_{\mathcal C(K,X)}$.
Since $\mathcal C(K,X)$ is ASQ there exists
$f\in S_{\mathcal C(K,X)}$ such that $\Vert f_i\pm f\Vert\leq 1+\varepsilon$
for all $i\in\{1,\ldots, n\}$. Pick $t\in K$ such that $f(t)\in S_X$, and
define $x:=f(t)$. Now
\begin{equation*}
   \Vert x_i\pm x\Vert =
   \Vert f_i(t)\pm f(t)\Vert
   \leq
   \Vert f_i\pm f\Vert\leq 1+\varepsilon
   \text{ for all } i\in\{1,\ldots,n\}.
\end{equation*}
So $X$ is ASQ\@.

(\ref{item:CKASQ-2}) $\Rightarrow$ (\ref{item:CKASQ-1}).
As $\mathcal C(K,X)= \mathcal C(K) \iten X$, we get the desired result from Corollary~\ref{estasqinject}.
\end{proof}

Next we turn to projective tensor products.
Given two Banach spaces $X$ and $Y$, recall that the
\textit{projective tensor product} of $X$ and $Y$, denoted by
$X\pten Y$, is the completion of $X\otimes Y$ under the norm given by
\begin{equation*}
   \Vert u \Vert :=
   \inf\left\{
      \sum_{i=1}^n  \Vert x_i\Vert\Vert y_i\Vert
      : u=\sum_{i=1}^n x_i\otimes y_i
      \right\}.
\end{equation*}
It is known that $B_{X\pten Y}=\overline{\conv}(B_X\otimes B_Y)
=\overline{\conv}(S_X\otimes S_Y)$ \cite[Proposition~2.2]{rya}.
Moreover, given Banach spaces $X$ and $Y$, it is well known that
$(X\pten Y)^*=L(X,Y^*)$ (see \cite{rya} for background).

From Corollary~\ref{estasqinject} we get examples of
projective tensor products that are octahedral.

\begin{corollary}\label{octaprojec}

Let $X$ and $Y$ be Banach spaces. Assume that $X$ is ASQ and $Y$ is Asplund.
If either $X^*$ or $Y^*$ has the approximation property,
then $X^* \pten Y^*$ is OH\@.

\end{corollary}

\begin{proof}
By assumption, we have $X^* \pten Y^*=(X \iten Y)^*$
(cf.\ e.g.\ \cite[Theorem~5.33]{rya}).
The result follows from
Corollary~\ref{estasqinject}.
\end{proof}

As noted in the Introduction, if a Banach space is LASQ, then it has the LD2P\@.
Given two Banach spaces $X$ and $Y$
it is known that $X\pten Y$ has the LD2P whenever $X$ has the LD2P
(see for example \cite[Theorem~2.7]{aln}).
Furthermore, $X \pten Y$ has the SD2P
whenever both $X$ and $Y$ have the SD2P \cite[Corollary~3.6]{blr}.
The next proposition gives us
examples of projective tensor product spaces which are LASQ\@.

\begin{proposition}\label{tensorc0ASQ}
  If $X$ is a Banach space, then $c_0 \pten X$ is LASQ\@.
  Moreover, $c_0 \pten X$ has the D2P\@.
\end{proposition}

\begin{proof}
  Let $Y := c_0 \pten X$.
  Let $u \in S_Y$ and $\varepsilon > 0$.
  Since $\conv(S_{c_0}\otimes S_X)$ is dense in $B_Y$
  we can find
  $v = \sum_{i=1}^n \lambda_i z_i\otimes x_i$
  such that $\|u-v\| \le \varepsilon$,
  where $z_i\in S_{c_0}$, $x_i\in S_X$
  and $\lambda_i \ge 0$ for each $i\in\{1,\ldots, n\}$
  with $\sum_{i=1}^n \lambda_i=1$.
  In particular, $\|v\| \ge 1 - \varepsilon$.
  We may assume that $z_1,\ldots, z_n$ have finite support,
  so we can find $m\in\mathbb N$ such that
  $n\geq m$ implies $z_i(n)=0$.
  Define
  \begin{equation*}
    \tilde{z}_i := \sum_{j=1}^{m-1} z_i(j)e_{m+j}
    \quad \text{and} \quad
    w := \sum_{i=1}^n \lambda_i \tilde{z}_i \otimes x_i.
  \end{equation*}
  We have
  \begin{equation*}
    \Vert v \pm w \Vert
    \leq \sum_{i=1}^n \lambda_i
    \left\Vert z_i \pm \tilde{z}_i \right\Vert
    \Vert x_i \Vert
    \leq 1
  \end{equation*}
  because $\|z_i \pm \tilde{z}_i\| = \max\{\|z_i\|,\|\tilde{z}_i\|\} = 1$
  since $z_i$ and $\tilde{z}_i$ have disjoint support
  for each $i\in\{1,\ldots, n\}$.
  Let $\theta$ be a permutation of $\mathbb{N}$ that
  swaps $k$ and $m+k$ for $k \in \{1,\ldots,m-1\}$.
  Let $\Phi_\theta$ be the isometry on $c_0$ defined by $\theta$.
  We have $\Phi_\theta(v)=w$ and if $T \in L(X,\ell_1) = Y^*$, then
  \begin{equation*}
    \langle \Phi_\theta^* T, v \rangle =
    \langle T, \Phi_\theta v\rangle = \langle T, w \rangle.
  \end{equation*}
  It follows that $\|w\| = \|v\|$
  and $\left \|u \pm \frac{w}{\|w\|}\right \| \le 1 + 2\varepsilon$,
  hence $Y$ is LASQ \cite[Proposition~2.1]{all}.
  
  Finally, let us prove that Y has the D2P\@. Let $u_0 \in B_Y$, $T_1,\ldots,T_k \in Y^*$
  and $\alpha > 0$. Consider the relatively weakly open neighborhood
  \begin{equation*}
    \mathcal{U} = \{
    y \in B_Y : |\langle y-u_0, T_j \rangle| < \alpha,
    \; j \in \{1,\ldots,k\}
    \}.
  \end{equation*}
  Then $\mathcal{U} \cap S_Y \neq \emptyset$.
  Let $u$ from the first part of the proof be in $\mathcal{U}$.
  With $\varepsilon$ small enough we may assume that $v$ is also
  in $\mathcal{U}$. Now $\{T_j x_i\}$, $i \in \{1,\ldots,n\}$
  and $j \in \{1,\ldots,k\}$, is a finite set of elements
  in $\ell_1$ so, replacing $e_{m+j}$ with $e_{N+j}$
  for $N$ big enough in the definition of $\tilde{z}_i$,
  we can make
  \begin{equation*}
    \langle w, T_j \rangle
    = \sum_{i=1}^n \lambda_i \langle \tilde{z}_i, T_j x_i\rangle
  \end{equation*}
  as small as we wish. Consequently, we may assume that $v \pm w \in \mathcal{U}$.
  Since
  \begin{equation*}
    \|v + w - (v-w)\| = \|2w\| \geq 2 - 2\varepsilon,
  \end{equation*}
  we conclude that $Y$ has the D2P\@.
\end{proof}

\begin{proposition}\label{ASQtensorproLASQ}
  Let $X$ and $Y$ be Banach spaces.
  Assume that there exists $f\in S_{Y^*}$ such that $n(Y^*,f)=1$.
  If $X$ is ASQ, then $X \pten Y$ is LASQ\@.
\end{proposition}

\begin{proof}
  Denote by $Z := X \pten Y$.
  Let $u \in S_Z$ and $\varepsilon > 0$.
  Since $\conv(S_X \otimes S_Y)$ is dense in $S_Z$
  we can find
  $v = \sum_{i=1}^n \lambda_i x_i\otimes y_i$
  such that $\|u-v\| \le \varepsilon$, where $x_i\in S_X, y_i\in S_Y$
  and $\lambda_i \ge 0$ for each $i\in\{1,\ldots, n\}$
  with $\sum_{i=1}^n \lambda_i=1$.
  Since $n(Y^*,f)=1$ we may assume (see \cite[Corollary~3.5]{abr})
  that $f(y_i)=1$ for each $i\in\{1,\ldots, n\}$, and hence
  $\left\Vert \sum_{i=1}^n \lambda_i y_i\right\Vert=1$.

  As $X$ is ASQ we can find $x\in S_X$ such that
  $\Vert x_i\pm x\Vert\leq 1+\varepsilon$
  for all $i\in\{1,\ldots, n\}$.
  Define $z:=x\otimes \sum_{i=1}^n \lambda_i y_i$, which is a norm one element.
  Now
  \begin{equation*}
    \| u\pm z\| \leq \|u - v\| + \|v \pm z\|
    \leq
    \varepsilon
    +
    \sum_{i=1}^n \lambda_i \Vert x_i \pm x \Vert \Vert y_i \Vert
    < 1 + 2\varepsilon.
  \end{equation*}
  We conclude that $Z$ is LASQ \cite[Proposition~2.1]{all}.
\end{proof}

\begin{rem}
From \cite[Theorem~3.1]{blr} it is immediate that
$X \pten Y$ has the SD2P in the above proposition.
However, we can not conclude that $X \pten Y$ is ASQ\@.
Indeed, let $X$ be ASQ and $Y=\ell_1^n$.
Then $X\pten Y=\ell_1^n(X)$ is not ASQ \cite[Lemma~5.5]{all}.
Given a Banach space $X$
it is not even clear whether $X\pten Y$ can be ASQ
for any Banach space $Y$ with $\dim(Y) \geq 2$.
\end{rem}

\section{Symmetric tensor product spaces}
\label{symtenprod}
Given $N \in \mathbb{N}$,
the \emph{($N$-fold) symmetric tensor product}, $\otimes^{s,N}X$, of a Banach space $X$
is the linear span of the tensors $x^N := x \otimes \cdots \otimes x$
in the $N$-fold tensor product $\otimes^N X$.
The \emph{($N$-fold) injective symmetric tensor product} of $X$,
denoted by $\widehat{\otimes}_{\varepsilon,s,N} X$,
is the completion of the space $\otimes^{s,N}X$
under the norm
\begin{equation*}
    \Vert u\Vert := \sup
    \left\{\left\vert
    \sum_{i=1}^n  x^*(x_i)^N\right\vert : x^*\in B_{X^*}
    \right\},
\end{equation*}
where $u:=\sum_{i=1}^n x_i^N$ (cf.\ e.g.\ \cite{flo}).
Note that $(\widehat{\otimes}_{\varepsilon,s,N} X)^*
= \mathcal P_I(^N X)$, the Banach space of $N$-homogeneous integral
polynomials on $X$.

We have seen in Theorem~\ref{theoctaopera}~(\ref{mejooctaopera}) that octahedrality is stable for non-symmetric injective tensor products. Now we shall establish the analogous result for the symmetric case.

\begin{theorem}\label{estaoctainjesyme}
Let $X$ be a Banach space and $N\in\mathbb N$.
If $X$ is OH, then
$\widehat{\otimes}_{\varepsilon,s,N} X$ is OH\@.
\end{theorem}

\begin{proof}
Denote by $Y:=\widehat{\otimes}_{\varepsilon,s,N} X$. Pick $u_1,\ldots, u_k\in S_Y$ and $\varepsilon>0$.
Assume, with no loss of generality, that
$u_i:=\sum_{j=1}^{n_i} x_{ij}^N$
for each $i\in\{1,\ldots,k\}$.
For each $i\in\{1,\ldots, k\}$ pick $x_i^*\in S_{X^*}$
such that $\sum_{j=1}^{n_i} x_i^*(x_{ij})^N>1-\varepsilon$.
As $X$ is OH
we can assure from Theorem~\ref{charaOHlanthe} the existence
of $y\in S_X$ and $y_1^*,\ldots, y_k^*\in X^*$ such that
\begin{align*}
   y_i^*(x_{ij}) &= x_i^*(x_{ij})
   \text{ for all } j\in\{1,\ldots, n_i\}, \\
   y_i^*(y) &= 1
   \ \mbox{and}\ \Vert y_i^*\Vert\leq 1+\varepsilon
   \text{ for all } i\in\{1,\ldots, k\}.
\end{align*}
Define $u:=y^N\in S_Y$. Now, given $i\in\{1,\ldots, k\}$, one has
\begin{equation*}
   \Vert u_i+u\Vert\geq \frac{ \sum_{j=1}^{n_i}y_i^*(x_{ij})^N
   + y_i^*(y)^N}{1+\varepsilon}
   =
   \frac{\sum_{j=1}^{n_i}x_i^*(x_{ij})^N+1}{1+\varepsilon}
   > \frac{2-\varepsilon}{1+\varepsilon}.
\end{equation*}
As $\varepsilon>0$ was arbitrary we conclude that $Y$ is OH, as desired.
\end{proof}

Given a Banach space $X$, denote by
$X^{(\infty}$ the completion of the linear space given by the union of
all even duals of $X$.
The above theorem allows us to improve \cite[Theorem~4.2]{ab},
where Acosta and Becerra Guerrero obtained $w^*$-D2P of spaces of homogeneous
integral polynomials when the Cunningham algebra $C(X^{(\infty})$ of $X^{(\infty}$
is infinite-dimensional.
As $\dim C(X^{(\infty}) = \infty$ implies that $X$ (and even
$X^{**}$) is OH (this follows from \cite[Theorem~3.4]{abl}),
we get the following result from Theorem~\ref{estaoctainjesyme}.

\begin{corollary}\label{integralpolyD2P}
Let $X$ be a Banach space.
If $C(X^{(\infty})$ is infinite-dimensional,
then $\mathcal P_I(^N X)$ has the $w^*$-SD2P  for each $N\in\mathbb N$.
\end{corollary}

Next we will weaken the assumptions of Theorem~\ref{estaoctainjesyme} and get examples of Banach spaces whose injective symmetric
tensor products are locally octahedral. The proof relies on the following geometrical characterization of weak octahedrality, whose
proof can be found in \cite{hlp}. 

\begin{theorem}[{\cite[Theorem~2.6]{hlp}}]\label{charaoctatela}
Let $X$ be a Banach space. Then $X$ is WOH if, and only if, for every finite-dimensional subspace $E$ of $X$, $\xast\in B_{\Xast}$,
and $\varepsilon>\nobreak0$,
there are $y\in S_X$ and $\xast_1,\xast_2\in\Xast$ with
$\|\xast_1\|,\|\xast_2\|\leq 1+\varepsilon$ satisfying
$\xast_1|_E=\xast_2|_E=\xast|_E$ and
$\xast_1(y)-\xast_2(y)>2-\varepsilon$.
\end{theorem}

\begin{proposition}\label{prop:XWOHsymtenLOH}

Let $X$ be a Banach space and let $N$ be an odd number.
If $X$ is WOH, then $\widehat{\otimes}_{\varepsilon,s,N}X$ is LOH\@.
In particular, $\mathcal P_I(^N X)$ has the $w^*$-LD2P\@.

\end{proposition}

\begin{proof}
Denote by $Y:=\widehat{\otimes}_{\varepsilon,s,N}X$. Consider $u:=\sum_{i=1}^k x_i^N\in S_Y$ and $\varepsilon>0$.
It is enough to prove the existence of $v\in S_Y$ such that
$\Vert u\pm v\Vert>\frac{1-\varepsilon+(1-\varepsilon)^N}{1+\varepsilon}$. To this aim consider $x^*\in S_{X^*}$ such that $1-\varepsilon<x^*(u)=\sum_{i=1}^n x^*(x_i)^N$.

Using Theorem~\ref{charaoctatela} we can find
$x_1^*,x_2^*\in X^*$ and $y\in S_X$ such that
\begin{align*}
  &x_1^*(x_i)=x_2^*(x_i)=x^*(x_i)
  \text{ for all } i\in\{1,\ldots, k\}, \\
  &x_1^*(y)-x_2^*(y)>2-\varepsilon, \; \text{and}\\
  &\Vert x_1^*\Vert\leq 1+\varepsilon,\Vert x_2^*\Vert\leq 1+\varepsilon.
\end{align*}
Now define $v:=y^N\in S_Y$. From $x_1^*(y)-x_2^*(y)>2-\varepsilon$
we conclude that $x_1^*(y)>1-\varepsilon$ and $-x_2^*(y)>1-\varepsilon$.
Consequently
\begin{align*}
   \Vert u+v\Vert &\geq
   \frac{\sum_{i=1}^k x_1^*(x_i)^N+x_1^*(y)^N}{1+\varepsilon} >
   \frac{\sum_{i=1}^k x^*(x_i)^N+(1-\varepsilon)^N}{1+\varepsilon} \\
   &> \frac{1-\varepsilon+(1-\varepsilon)^N}{1+\varepsilon}.
\end{align*}
On the other hand
\begin{align*}
   \Vert u-v\Vert &\geq
   \frac{\sum_{i=1}^k x_2^*(x_i)^N-x_2^*(y)^N}{1+\varepsilon}
   >
   \frac{\sum_{i=1}^k x^*(x_i)^N+(1-\varepsilon)^N}{1+\varepsilon} \\
   &>
   \frac{1-\varepsilon+(1-\varepsilon)^N}{1+\varepsilon}.
\end{align*}
Hence we have $\|u \pm v\| >
\frac{1-\varepsilon+(1-\varepsilon)^N}{1+\varepsilon}$, so $Y$ is LOH
\end{proof}

\begin{rem}
We do not know whether one can get D2P in the above proposition.
Neither do we know whether the above proposition holds for an even number $N$.
\end{rem}

We pass now to projective symmetric tensor products.
Given a Banach space $X$, we define the
\textit{($N$-fold) projective symmetric tensor product} of $X$, denoted by
$\widehat{\otimes}_{\pi,s,N} X$, as the completion of the space
$\otimes^{s,N}X$ under the norm
\begin{equation*}
   \Vert u\Vert:=\inf
   \left\{
      \sum_{i=1}^n \vert \lambda_i\vert \Vert x_i\Vert^N :
      u:=\sum_{i=1}^n \lambda_i x_i^N, n\in\mathbb N, x_i\in X
   \right\}.
\end{equation*}
The dual, $(\widehat{\otimes}_{\pi,s,N} X)^*=\mathcal P(^N X)$, is
the Banach space of $N$-homogeneous continuous polynomials on $X$ (see \cite{flo} for background).

It is known that $\widehat{\otimes}_{\pi,s,N}X$ has the SD2P provided
the Banach space $X$ is ASQ \cite[Theorem~3.3]{blr3}.
The proof of this result relies heavily
on the fact that the sequences involved in the definition of ASQ
can be chosen to be $c_0$-sequences (see \cite[Lemma~2.6]{all}).
If one tries to copy the idea of the ASQ proof in the WASQ
setting it does not work because WASQ Banach spaces do not
have to contain any isomorphic copy of $c_0$
\cite[Proposition~3.5]{all}.
Consequently, in order
to connect WASQ spaces with diameter two properties in projective
symmetric tensor products, we shall need an extra assumption.

A Banach space $X$ has the
\textit{Dunford-Pettis property} if every weakly compact operator
$T:X\longrightarrow Y$ is weak-to-norm sequentially continuous,
i.e., whenever $x_n \rightarrow x$ in $X$
weakly then $Tx_n \rightarrow Tx$ in norm in $Y$
(cf.\ e.g.\ \cite[Proposition~5.4.2]{alka}).
It is known that every continuous polynomial on a Banach space having
the Dunford-Pettis property is weakly sequentially continuous \cite[Proposition~2.34]{din}.

\begin{proposition}\label{dunfordpettistensoprosym}

Let $X$ be a Banach space with the Dunford-Pettis property.
If $X$ is WASQ, then $\widehat{\otimes}_{\pi,s,N}X$
has the LD2P for each $N\in\mathbb N$.

\end{proposition}

\begin{proof}
Let $N\in\mathbb N$. It is enough to prove that
$\left(\widehat{\otimes}_{\pi,s,N}X \right)^*=\mathcal P(^NX)$ is LOH\@.
To this aim pick $P\in S_{\mathcal P(^NX)}$ and $\varepsilon>0$.
Pick $x\in S_X$ such that $P(x)>1-\varepsilon$.

Since $X$ is WASQ we can find a sequence $\{y_n\}$ in
$B_X$ such that $y_n \rightarrow 0$ weakly, $\Vert y_n\Vert \rightarrow 1$ and
$\Vert x\pm y_n\Vert \rightarrow 1$.
As $\{y_n\}$ is a weakly-null sequence and $X$ has
the Dunford-Pettis property we conclude that $P(y_n) \rightarrow 0$
\cite[Proposition~2.34]{din}.
Consequently $P(x\pm y_n) \rightarrow P(x)>1-\varepsilon$
\cite[Lemma~1.1]{fj}. So we can find $n\in\mathbb N$ big enough to ensure
\begin{equation*}\label{condipoliwasqprop}
   P(x\pm y_n)>1-\varepsilon,
\end{equation*}
$\Vert x\pm y_n\Vert\leq 1+\varepsilon$ and $\Vert y_n\Vert>1-\varepsilon$.
Choose $f\in S_{X^*}$ such that $f(y_n)>1-\varepsilon$. Now
\begin{equation*}
   1+\varepsilon\geq \Vert y_n \pm x \Vert \geq f(y_n \pm x)
   \ge 1-\varepsilon \pm f(x),
\end{equation*}
which implies that $\vert f(x)\vert < 2\varepsilon$.

Now we shall argue by cases:
If $N$ is odd, define $Q(x):=f(x)^N$ for each $x\in X$,
which is a norm one $N$-homogeneous polynomial. Now
\begin{align*}
   \Vert P\pm Q\Vert &\geq (P\pm Q)
   \left(\frac{x\pm y_n}{\Vert x\pm y_n\Vert}\right)
   =
   \frac{P(x\pm y_n)\pm Q(x\pm y_n)}{\Vert x\pm y_n\Vert^N} \\
   &>
   \frac{1-\varepsilon\pm (f(x)\pm f(y))^N}{(1+\varepsilon)^N}
   =
   \frac{1-\varepsilon+(f(y)\pm f(x))^N}{(1+\varepsilon)^N} \\
   &>
   \frac{1-\varepsilon-(1-3\varepsilon)^N}{(1+\varepsilon)^N}.
\end{align*}
If $N$ is even, pick $g\in S_{X^*}$ such that $g(x)=1$,
which implies that $\vert g(y_n)\vert<\varepsilon$.
Define $Q(z)=g(z)f(z)^{N-1}$ for all $z\in X$,
which is a norm one polynomial. Then
\begin{align*}
   \Vert P\pm Q\Vert &\geq (P\pm Q)
   \left(\frac{x\pm y_n}{\Vert x\pm y_n\Vert}\right)
   =
   \frac{P(x\pm y_n)\pm Q(x\pm y_n)}{\Vert x\pm y_n\Vert^N} \\
   &>
   \frac{1-\varepsilon\pm (g(x)\pm g(y))(f(x)\pm f(y))^{N-1}}{(1+\varepsilon)^{N}}.
\end{align*}
Similar estimates to the ones above allow us to conclude that
\begin{equation*}
   \frac{1-\varepsilon\pm (g(x)\pm g(y))(f(x)\pm f(y))^{N-1}}{(1+\varepsilon)^{N}}
   >
   \frac{1-\varepsilon+(1-\varepsilon)(1-3\varepsilon)^{N-1}}{(1+\varepsilon)^N}.
\end{equation*}
In any case, as $\varepsilon>0$ was arbitrary,
we conclude that $\mathcal P(^NX)$ is LOH\@.
\end{proof}

In order to exhibit examples of Banach spaces where the above proposition
applies we shall prove the following result about inheritance of WASQ
to subspaces.
Note that it extends the result for the D2P from
\cite[Theorem~2.2]{blr4}, where it is proved that D2P is
inherited by finite codimensional subspaces. In addition,
the theorem below seems to be the only known
result about inheritance of WASQ by subspaces.

\begin{theorem}\label{subespWASQ}
Let $X$ be WASQ and let $Y\subseteq X$ be a closed subspace.
If $X/Y$ has the Schur property, then $Y$ is WASQ\@.
\end{theorem}

\begin{proof}
Pick $y\in S_Y\subseteq S_X$. As $X$ is WASQ we can find $\{x_n\}$
a sequence in $B_X$ such that
$\Vert y\pm x_n\Vert \rightarrow 1$,
$\Vert x_n\Vert \rightarrow 1$, and
$x_n \rightarrow 0$ weakly.

Consider the quotient map $\pi:X\longrightarrow X/Y$, which is a
weak-to-weakly continuous map. We get
$\pi(x_n) \rightarrow 0$ weakly and, by the Schur property of $X/Y$,
$\pi(x_n) \rightarrow 0$ even in norm. Hence,
for each $n\in\mathbb N$, we can find $y_n\in B_Y$ such that
\begin{equation*}
   \Vert y_n-x_n\Vert<\Vert \pi(x_n)\Vert+\frac{1}{n}.
\end{equation*}
We shall prove that $\{y_n\}\subseteq B_Y$ satisfies our requirements.
On the one hand
\begin{equation*}
   1\geq \Vert y_n\Vert
   \geq
   \Vert x_n\Vert-\Vert y_n-x_n\Vert
   \text{ for all } n\in\mathbb N,
\end{equation*}
so $\Vert y_n\Vert \rightarrow 1$. On the other hand
\begin{equation*}
  \Vert y\pm x_n\Vert-\Vert x_n-y_n\Vert\leq \Vert y\pm y_n\Vert\leq \Vert y\pm x_n\Vert+\Vert x_n-y_n\Vert
\end{equation*}
holds for each $n\in\mathbb N$, hence $\Vert y\pm y_n\Vert \rightarrow 1$.
Finally, as $x_n \rightarrow 0$
weakly and $x_n-y_n \rightarrow 0$ weakly
(even in norm) we conclude that $y_n \rightarrow 0$ weakly.
By definition, $Y$ is WASQ, as desired.
\end{proof}

\begin{rem}
It is clear that the above proof also shows that given a Banach space
$X$ which is ASQ and $Y\subseteq X$ a closed subspace such that $X/Y$
has the Schur property, then $Y$ is ASQ\@. However, this can be
deduced from a more general result \cite[Theorem~3.6]{abrahamsen},
where it is proved that ASQ is inherited by closed subspaces whose
quotient space does not contain any isomorphic copy of $c_0$.
\end{rem}

In \cite{ab2} it was shown that the $N$-fold
projective symmetric tensor product of $L_1(\mu)$,
$\mu$ a $\sigma$-finite measure,
has the LD2P for each $N\in\mathbb N$.
Now we can combine Proposition~\ref{dunfordpettistensoprosym}
and Theorem~\ref{subespWASQ} to extend Theorem~3.1 of that paper,
by considering some closed subspaces of $L_1(\mu)$.
The proof relies on the fact that $L_1(\mu)$ is WASQ,
a fact that can be proved using the ideas in \cite[Lemma~3.3]{kub}.

\begin{corollary}
Let $X$ be a complemented subspace of $L_1(\mu)$, where $(\Omega,\Sigma,\mu)$
is a measure space and $\mu$ is a $\sigma$-finite measure. If $L_1(\mu)/X$
is isomorphic to $\ell_1$  then  $\widehat{\otimes}_{\pi,s,N}X$ has the LD2P for each $N\in\mathbb N$.
\end{corollary}

\begin{proof}
As $L_1(\mu)/X$ is isomorphic to $\ell_1$ then $L_1(\mu)/X$ has the
Schur property. Consequently, $X$ is WASQ by Theorem~\ref{subespWASQ}.
In addition, $X$ has the Dunford-Pettis property because $L_1(\mu)$
has the Dunford-Pettis property \cite[Theorem~5.4.5]{alka}
and $X$ is complemented in $L_1(\mu)$. So the result holds as
an immediate application of Proposition~\ref{dunfordpettistensoprosym}.
\end{proof}

\begin{rem}
Note that given $L_1(\mu)$ as in the above corollary and given
$X\subseteq L_1(\mu)$ an infinite-dimensional and complemented subspace,
there are conditions on $L_1(\mu)/X$ which guarantee that it is isomorphic
to $\ell_1$ such as having an unconditional basis or
the Radon-Nikod\'{y}m property (see \cite[p.~122]{alka}).
Moreover, it is conjectured that each infinite-dimensional complemented
subspace of $L_1(\mu)$ is isomorphic either to $L_1(\mu)$ or to
$\ell_1$ (see e.g. \cite[Conjecture~5.6.7]{alka}).
If this conjecture were correct, the above corollary would have a wide
range of applications.
\end{rem}

\section{Some remarks and open questions}
\label{remarks}
In this section we will pose some open questions related to our main results.
In light of Theorem~\ref{theoctaopera} the following
question arises.

\begin{question}

Let $X$ and $Y$ be Banach space and $H\subseteq L(X^*,Y)$ a closed subspace such
that $X\otimes Y\subseteq H$ and that each element of $H$ is weak$^*$-to-weakly continuous.
Is $H$ OH whenever $X$ is OH?

\end{question}

Theorem~\ref{theoctaopera}~(\ref{corounillocta}) provides
a partial positive answer. On the other hand, in
\cite[Theorem~4.2]{kkw} an example is given of a complex two-dimensional
Banach space $E$ such that $L_1^\mathbb C([0,1])\iten E$ fails to
have the Daugavet property and, consequently, is a natural candidate
for a negative answer to the above question.

In Section~\ref{tensorproducts} we saw conditions on Banach spaces
$X$ and $Y$ which
ensure that certain subspaces of $L(X^*,Y)$ are LOH or OH\@. So it is natural to wonder

\begin{question}
Let $X$ and $Y$ be Banach spaces and $H\subseteq L(X^*,Y)$ be a closed subspace
such that $X\otimes Y\subseteq H$. When is $H$ WOH?
\end{question}
We note that from Proposition~\ref{tensorc0ASQ}
we get that $L(X,\ell_1) = (c_0 \pten X)^*$ is WOH for any Banach
space $X$.
Similarly, for any $K$ infinite compact Hausdorff space,
$L(X,C(K)^*) = (C(K) \pten X)^*$ is WOH for any Banach space $X$
by \cite[Theorem~4.1]{abr}.

Corollary~\ref{estasqinject} says that the injective tensor product is ASQ
provided one of the factors is ASQ\@.
For the symmetric case we ask:

\begin{question}

Let $X$ be a Banach space.
Is $\widehat{\otimes}_{\varepsilon,s,N}X$ ASQ for each $N\in\mathbb N$
whenever $X$ is ASQ?
\end{question}

Concerning octahedrality in projective tensor products
we have only seen one result, Corollary~\ref{octaprojec}.
However, there is a wider class of examples of OH projective tensor spaces.

\begin{example}
Let $X$ and $Y$ be Banach spaces.

\begin{enumerate}
\item\label{questOH:0}
  We saw in the example following Corollary~\ref{estasqinject}
  that given $1< p\leq q<\infty$ then $K(\ell_p,\ell_q)$
  is ASQ\@.
  By \cite[Proposition~2.5]{all} the dual space $\ell_p\pten
  \ell_{q^*}$ is OH, where $\frac{1}{q}+\frac{1}{q^*}=1$.
\item\label{questOH:1}
  If $X = L_1(\mu)$ then $X \pten Y = L_1(\mu,Y)$ is OH
  by a straightforward computation.

\item\label{questOH:2}
   If $X=\ell_1(I)$ for a suitable infinite set $I$, then $X\pten Y$ is OH\@.

   \noindent
   This follows from the identification $\ell_1(I) \pten Y=\ell_1(I,Y)$
   \cite[Example~2.6]{rya} and the fact that infinite $\ell_1$-sums of
   Banach spaces are always OH\@.
\item\label{questOH:3}
   If $C(X)$ is infinite-dimensional, then $X\pten Y$ is OH\@.

   \noindent
   This follows from the fact that
   $(X \pten Y)^* = L(Y,X^*)$ has an infinite-dimensional
   centralizer \cite[Lemma~VI.1.1]{hww}
   and thus it has the SD2P \cite[Proposition~3.3]{abl}.

\item\label{questOH:4}
   If $H$ is a Hilbert space, then $\mathcal F(H)\pten H$ is OH,
   where $\mathcal{F}(H)$ is the Lipschitz-free space over $H$.

   \noindent
   This follows from \cite[Theorem~2.5]{blr5} and from
   the fact that the space of Lipschitz functions on $H$
   which vanish at $0$ under the Lipschitz norm
   can identified with $L(\mathcal F(H),H)$.
\end{enumerate}
\end{example}

In view of (\ref{questOH:1}),(\ref{questOH:2}),
(\ref{questOH:3}) and (\ref{questOH:4}) of the above example we wonder

\begin{question}
Let $X$ and $Y$ be Banach spaces. Is $X\pten Y$ OH whenever $X$ is OH?
\end{question}

\end{document}